\documentclass[12pt]{article}
\usepackage{amssymb}
\usepackage{amsthm}
\usepackage{dsfont}
\usepackage[a4paper,headsep=0cm, hmargin={3cm,3cm},vmargin={3cm,3cm}]{geometry}

\renewenvironment{itemize}
  {\begin{list}{$\triangleright$}{%
   \setlength{\parskip}{0mm}
   \setlength{\topsep}{2mm}
   \setlength{\rightmargin}{0mm}
   \setlength{\listparindent}{0mm}
   \setlength{\itemindent}{0mm}
   \setlength{\labelwidth}{5mm}
   \setlength{\itemsep}{2mm}
   \setlength{\parsep}{0mm}
   \setlength{\partopsep}{0mm}
   \setlength{\labelsep}{3mm}
   \setlength{\leftmargin}{8mm}
   }}{%
   \end{list}\vspace*{-1.3mm}}

\def\E{\exists}
\def\A{\forall}
\def\models{\vDash}

\def\odim{\textrm{o-dim}}

\def\kdim{\textrm{k-dim}}

\def\eqntp{{\rm eqn\textrm{-}tp}}

\def\sm{\smallsetminus}

\def\IMP{\Rightarrow}

\def\imp{\rightarrow}

\def\iff{\leftrightarrow}

\def\<{\langle}
\def\>{\rangle}
\def\0{\varnothing}
\def\theta{\vartheta}
\def\phi{\varphi}
\def\epsilon{\varepsilon}
\def\emph#1{{\boldmath\bfseries #1}}
\def\ssf#1{\textsf{\footnotesize #1}}

\newtheoremstyle{mio}
     {2\parskip}
     {0mm}
     {\sl}
     {}
     {\bfseries}
     {}
     {3mm}
     {\thmnumber{#2}\hskip2mm\thmname{#1}\thmnote{\bfseries{} #3}}

\newtheoremstyle{ex}
     {2\parskip}
     {0mm}
     {}
     {}
     {\bfseries}
     {}
     {3mm}
     {\thmnumber{#2}\hskip2mm\thmname{#1}\thmnote{\bfseries{} #3}}

\theoremstyle{mio}
\newcounter{thm}
\newtheorem{theorem}[thm]{Theorem}
\newtheorem{corollary}[thm]{Corollary}
\newtheorem{lemma}[thm]{Lemma}
\newtheorem{fact}[thm]{Fact}

\newtheorem{definition}[thm]{Definition}

\theoremstyle{ex}
\newtheorem{example}[thm]{Example}
\newtheorem{observation}[thm]{Observation}

\def\QED{\nolinebreak[4]\hfill\rlap{\ \ $\Box$}\medskip}
\renewenvironment{proof}[1][Proof]%
{\begin{trivlist}\item[\hskip\labelsep {\bf #1}]}
{\QED\end{trivlist}}

\begin{document}
\raggedbottom

\begin{center}
\begin{huge}
Krull dimension of types\\[4mm] in a class of first-order theories.
\end{huge}
\par\bigskip\bigskip\bigskip
Domenico Zambella
\par\bigskip\bigskip\bigskip
\begin{minipage}{70ex}\small
We study a class of first-order theories whose complete quantifier-free types with one free variable either have a trivial positive part or are isolated by a positive quantifier-free formula---plus a few other technical requirements. The theory of vector spaces and the theory fields are examples. We prove the amalgamation property and the existence of a model-companion. We show that the model-companion is strongly minimal. We also prove that the length of any increasing sequence of prime types is bounded, so every formula has finite Krull dimension.\bigskip

\end{minipage}
\end{center}
\bigskip
\section{Introduction}

Krull-minimal theories are defined in Definition~\ref{DT} below. The main requirement is that every complete quantifier-free type with one free variable either has a trivial positive part or it is isolated by a positive quantifier-free formula. This means that the formula $x=x$ has Krull-dimension $\le 1$, as defined in Section~\ref{noetherianity} below. 

We show that Krull-minimal theories have the amalgamation property in Theorem~\ref{AP} and that they are model-companionable in Theorem~\ref{richsaturated}.  In Corollary~\ref{stronglyminimal} we show that the model-companion of a Krull-minimal theory is a strongly minimal theory. This reproduces in general the usual aguments used to prove elimination of quantifiers for vector spaces, torsion-free divisible groups, fields, see e.g.\@ \cite{m1} and/or \cite{m2}.  In Section~\ref{noetherianity} we consider two notions of dimension. We prove that in a Krull-minimal theory the length of any increasing sequence of prime types is bounded by the maximal degree of transcendence of its solutions. So every formula has finite Krull dimension.

In the past there has been some interest in first-order theories with formulas that satisfy descending chain conditions, see \cite{ps}, \cite{j}, and references cited therein---as we found out when the final draft of this paper was ready. Our general setting is different, the most relevant difference is that we require strong properties to hold for one variable formulas and  pose no requirement on formulas with many variables. Still, a question considered in this paper is not far in spirit from a question asked in \cite{ps}. In \cite{ps} it is asked whether all $1$-equational theories are $n$-equational---see Definition~2.1 in \cite{j}. The aim of this paper is to try and understand which descending chain conditions can be deduced from good behaviour of formulas in one variable.

I wish to thank Anand Pillay for pointing out an inconsistency in a draft version of this paper and for drawing my attention to \cite{ps}. Also, I  would like to express my gratitude to Alberto Albano for valuable discussions.

\section{Krull-minimal theories}
\def\ftp{\textrm{ftp}}

Throughout this paper $T$ is a consistent theory: we shall say \emph{model\/} for \textit{model of\/} $T$, \emph{consistent\/} for consistent \textit{modulo\/} $T$, and \emph{complete\/} for complete \textit{modulo\/} $T$. The letters $M$, $N$, etc.\@ denote models and $A$, $B$, etc.\@ denote subsets of models. We assume that every model is contained in an infinite one and that $T$ fixes the \emph{characteristic of the models\/} i.e.\@ all substructures generated by the empty sets are isomorphic. In other words, $T$ is complete for quantifier-free sentences.

We say $A$-type, $A$-formula for type, respectively formula, over $A$. We write \emph{$\vdash_{\!A} q(z)$}, where $q(z)$ is an $A$-type, when $\A z\, q(z)$ holds in every model containing the substructure generated by $A$. We omit the subscript when $A$ is empty. The expression \emph{$p(z)\vdash_{\!A} q(z)$\/} abbreviates $\vdash_{\!A} p(z)\imp q(z)$. Let $a$ be a possibly infinite tuple of elements of some model and let $p(w)$ the quantifier-free type of $a$; we shall use several times without further mention that $\vdash_{\!a} q(a,z)$ is equivalent to $p(w)\vdash q(w,z)$. The notation introduced in this paragraph is less common. From next section, when we know that a model-companion of $T$ exists, one could substitute $\vdash_{\!A} q(z)$ with $U\models\A z\,q(z)$, where $U$ is some large saturated existentially-closed model of $T$. 

An \emph{equational formula\/} is a formula that contains only the connectives $\top$, $\bot$, $\wedge$ and $\vee$. An \emph{equational type\/} is a set of equational formulas. Let $p(z)$ be an $A$-type (not necessarily equational). We say that $p(z)$ is \emph{trivial over $A$\/} if $\vdash_{\!A}p(z)$. We say it is \emph{consistent over $A$\/} if $\nvdash_{\!A} \neg p(z)$, that is, if it is realized in some model containing $A$ or, in other words, if it is consistent with the quantifier-free type of $A$. We say that $p(z)$ is \emph{maximal over $A$\/} if it is consistent over $A$ and for every equational $A$-formula $\phi(z)$ either $p(z)\vdash_{\!A}\phi(z)$ or $p(z)\vdash_{\!A}\neg\phi(z)$. It is \emph{prime over $A$\/} if it is consistent over $A$ and for every pair of equational $A$-formulas $\phi(z)$ and $\psi(z)$,

\hfil$p(z)\vdash_{\!A}\phi(z)\vee\psi(z)\ \ \IMP\ \ p(z)\vdash_{\!A}\phi(z)\ \textrm{ or }\ p(z)\vdash_{\!A}\psi(z)$.

As expected, maximal implies prime. The specification `over $A$' will be dropped when $A$ is empty or clear from the context. We say that $p(z)$ is \emph{principal\/} if there is an $A$-formula $\phi(z)$ such that $\vdash_{\!A}\;p(z)\iff\phi(z)$. By compactness we can always assume that $\phi(z)$ is a conjunction of formulas in $p(z)$. When $p(z)$ is maximal, we may say \emph{isolated\/} for principal. The terminology could mislead the readers that uses ring of polynomials to guide their intuition. In fact, principal equational types correspond to finitely generated ideals in the ring of polynomials \textit{not\/} to principal ideals (which do not have an analog in our setting).

\textbf{\boldmath Throughout this paper $x$ denotes a single variable.}

\begin{definition}\label{DT}
We say that the theory $T$ is \emph{Krull-minimal\/} if for every tuple $a$,
\begin{itemize}
\item[D0] the trivial type $x=x$ is prime over $a$;
\item[D1] for every consistent equational formula $\zeta(a,x)$ there is a quantifier-free formula $\theta(z)$ such that $\vdash_{\!a}\theta(a)$ and $\zeta(b,x)$ is consistent for every $b$ such that $\vdash_{\!b}\theta(b)$;
\item[D2] every equational formula $\zeta(a,x)$ that is consistent over $a$, is consistent over any $B$ containing $a$; 
\item[D3] every non-trivial prime equational type $p(a,x)$ is maximal and principal.
\end{itemize} 
We say that $T$ is \emph{locally Krull-minimal} if these conditions hold when $a$ is a finite tuple.
\end{definition}

The heuristic is as follows: grosso modo axioms \ssf{D2} and \ssf{D3} ensure the amalgamation property of models and axiom \ssf{D1} the existence of a model-companion. Axiom \ssf{D0} is introduced for a smooth and non-trivial theory of dimension. It is only necessary for the definition of locally Krull-minimal theory in fact, when infinite tuples of parameters are allowed, it follows easily from \ssf{D1}. All results in this section are independent of \ssf{D0}.

\begin{example}
The theory of integral domains is a Krull-minimal theory. In fact, observe that equational formulas in the language of rings are systems of equations. Then \ssf{D0}, \ssf{D1} and \ssf{D2} are obvious. To prove \ssf{D3}, let $t(a,x)$ be a polynomial of minimal deree among those such that $p(a,x)\,\vdash_{\!a}\,t(a,x)=0$. The maximality of the formula $t(a,x)=0$ is an easy consequence of B\'ezout identity (in the field of fractions generated by $a$).\QED
\end{example}

\begin{example}
The theory of modules over a fixed integral domain is a Krull-minimal theory.\QED
\end{example}

The following examples show that the situation is different for locally Krull-minimal theories.

\begin{example}
Consider a language that contains only a binary relation $r(x,y)$. Let $T_{\sf a}$, $T_{\sf b}$, and $T_{\sf c}$ be the theories that axiomatize the models that are
\begin{itemize}
\item[a] disjoint union of tournaments and, respectively,
\item[b] disjoint union of dense linearly ordered sets without endpoints. 
\item[c] disjoint union of tournaments and linearly ordered sets. 
\end{itemize}
It is easy to check that these are locally Krull-minimal theories. Observe that $T_{\sf a}$ is an unstable simple theory, $T_{\sf b}$ is an unstable theory without the independence property, and  $T_{\sf c}$ is neither.\QED
\end{example}

Only one of the requirements in Definition~\ref{DT} may fail in a locally Krull-minimal theory for some infinite tuple of parameters $a$: some non-trivial prime equational type $p(a,x)$ is non-isolated. In fact, infinite tuples of parameters are irrelevant for \ssf{D0}, \ssf{D1} and \ssf{D3} while the observation below ensure that first claim in \ssf{D3} holds also for infinite tuples of parameters.

\begin{observation}\label{notstrongly}
Let $T$ be a locally Krull-minimal theory and let $p(x)$ be a non-trivial prime equational $A$-type, for some possibly infinite set of parameters $A$. Then, by compactness, $p(x)$ is maximal.\QED
\end{observation}

We will often think of prime equational types as the positive part of a maximal quantifier-free type: this is precisely stated in point \ssf{c} of the following observation. We need some notation: for every type $p(z)$ we define

\hspace*{14ex}\llap{\emph{$p^\bullet(z)$}}$\ \ :=\ \ \Big\{\,\neg\xi(z)\ \ :\  \ \xi(z)\textrm{ is an equational }A\textrm{-formula and }  p(z)\nvdash_{\!A}\xi(z)\,\Big\}$,

\hspace*{14ex}\llap{\emph{$p^\circ(z)$}}$\ \ :=\ \ \Big\{\,\phantom{\neg}\xi(z)\ \ :\ \  \xi(z)\textrm{ is an equational }A\textrm{-formula and }  p(z)\vdash_{\!A}\xi(z)\,\Big\}$.

Note the dependency on $A$, however, as $A$ will be always clear from the context, we omit it from the notation. When $p(z)$ is trivial then $p^\bullet(z)$ is called the \emph{transcendental $A$-type\/} and is denoted by \emph{$o(z/A)$}. In general, it need not be a consistent type but, when it is consistent, it is also maximal. The following type is called the \emph{equational type of $b$ over $A$}:

\hspace*{14ex}\llap{\emph{$\eqntp(b/A)$}}$\ \ :=\ \ \Big\{\,\,\phantom{\neg}\xi(z)\ \ :\ \ \ \xi(z)\textrm{ is an equational }A\textrm{-formula and }\vdash_{\!A,b}\xi(b)\,\Big\}$,

As usual, when $A$ is empty we omit if from the notation.

\begin{fact}\label{observation*} Let $p(z)$ be any $A$-type. The following are equivalent:
\begin{itemize}
\item[a] $p(z)$ is prime over $A$;
\item[b] $p^\bullet(z)\,\wedge\, p(z)$ is consistent, and consequently maximal, over $A$.
\item[c] $p^\circ(z)=\eqntp(b/A)$ for some $b$.
\end{itemize}
\end{fact}
\begin{proof}
The implications $\ssf{a}\IMP\ssf{b}\IMP\ssf{c}$ are clear by compactness. The implication $\ssf{c}\IMP\ssf{a}$ amounts to claiming that $p^\circ(z)=\eqntp(b/A)$ is prime over $A$. Suppose $p^\circ(z)\vdash_{\!A}\phi(z)\vee\psi(z)$ for some equational $A$-formulas $\phi(z)$ and $\psi(z)$. Let $M$ be any model containing $A,b$, then $M\models\phi(b)\vee\psi(b)$, so $M\models\phi(b)$ or $M\models\psi(b)$. Observe that if $\phi(z)$ is an equational $A$-formula and $M\models\phi(b)$ for \textit{some\/} model $M$ containing $A,b$ then $M\models\phi(b)$ for \textit{all\/} models $M$ containing $A,b$ and $\phi(z)$ belongs to $\eqntp(b/A)$. For this it is essential to recall that by `containing' $A,b$ we understand `containing the substructure generated by' $A,b$. Then $p^\circ(z)\vdash_{\!A}\phi(z)$ or $p^\circ(z)\vdash_{\!A}\psi(z)$.
\end{proof}

\begin{observation}\label{observation2}
We can rephrase condition \ssf{D3} above as follows: for every complete quantifier-free type $p(a,x)$ either $p^\circ(a,x)$ is trivial or there is an equational formula $\xi(z,x)$ such that $\vdash_{\!a}\,\xi(a,x)\,\iff\,p(a,x)$.\QED
\end{observation}

Prime types play the role of complete types when we restrict the attention to equational formulas. Precisely, we have the following fact.

\begin{fact}\label{juisn}
Let $q(z)$ be an equational $A$-type and let $P$ be the set of prime equational $A$-types $p(z)$ such that $p(z)\vdash_{\!A} q(z)$. Then

\hfil$\displaystyle\vdash_{\!A}\ q(z)\ \iff \bigvee_{p(z)\in P}p(z)$,

where if $P$ is empty the disjunction is $\bot$.
\end{fact}

\begin{proof}
It follows from \ssf{c} of Fact~\ref{observation*}.
\end{proof}

\begin{fact}\label{nonprime}
Let $T$ be a Krull-minimal theory and let $A$ be an arbitrary set of parameters (alternatively: $T$ locally Krull-minimal and $A$ finite). Let $p(x)$ be a non-trivial equational $A$-type. Then there are some equational $A$-formulas $\xi_1(x),\dots,\xi_n(x)$ that are maximal and such that

\hfil$\displaystyle\vdash_{\!A}\ p(x)\ \iff \bigvee^n_{i=1}\xi_i(x)$.

\end{fact}

\begin{proof}
By Fact~\ref{juisn}, axiom~\ssf{D3}, and compactness.
\end{proof}

\begin{theorem}\label{minimal}
Let $T$ be a Krull-minimal theory and let $\phi(x)$ be a non-trivial equational $A$-formula. Then $\vdash_{\!A}\E^{<n}x\,\phi(x)$ for some $n$.
\end{theorem}

\begin{proof}
Suppose for a contradiction that $\nvdash_{\!A}\phi(x)$ and that there is an infinite set $B$ such that $\vdash_{\!A,b}\phi(b)$ for every $b\in B$. Apply Fact~\ref{nonprime} to obtain

\hfil $\displaystyle\vdash_{\!A,B}\  \phi(x)\ \iff\ \bigvee^n_{i=1}\xi_i(x)$,

for some equational $A,B$-formulas $\xi_1(x),\dots,\xi_n(x)$ that are maximal over $A,B$. One of these formulas, say $\xi_i(x)$, is satisfied by two distinct elements in $B$, say $b_1$ and $b_2$. But $\xi_i(x)$ is maximal so it implies both $x=b_1$ and $x=b_1$, a contradiction.
\end{proof}

\begin{theorem}\label{AP}
Let $T$ be a locally Krull-minimal theory. The class of models of $T$ has the amalgamation property (over sets).
\end{theorem}

\begin{proof}
Let $A=M\cap N$ be a substructure of both $M$ and $N$. We show that there is a model $N'$ containing $N$ and an embedding $f':M\imp N'$ that fixes $A$.  Clearly it suffices to show that for any element $b\in M$ there are a model $N'$ containing $N$ and some $b'\in N'$ such that $N,b\equiv_{A,\rm qf} N',b'$. Let $a$ be a tuple that enumerates $A$. Let $p(z,x)=\eqntp(a,b)$, which is prime by Fact~\ref{observation*}.\ssf{b}. Assume first that $p(a,x)$ is trivial. Then $b$ realizes $o(x/A)$ in $M$. Take any $N',b'$ realizing $o(x/N)$, which is consistent by \ssf{D0}. Now assume instead that $p(a,x)$ is non-trivial. By \ssf{D2}, the type $p(a,x)$ is consistent over $N$, so there is an $N'$ containing $N$ and $b'\in N'$ realizing $p(a,x)$. By Observation~\ref{notstrongly}, $p(a,x)$ is maximal over $A$, so $N,b\equiv_{A,\rm qf} N',b'$.
\end{proof}

\begin{theorem}\label{richsaturated}
Let $T$ be a locally Krull-minimal theory. Then $T$ has a model-companion that admits elimination of quantifiers.
\end{theorem}

\begin{proof}
Let $T_c$ be the theory of the existentially closed models of $T$ and let $M$ and $N$ be two $\omega$-saturated models of $T_c$, let $a$ and $d$ be finite tuples in $M$, respectively, $N$ and such that $M,a\equiv_{\rm qf}N,d$. Let $c$ be an element of $M$, we show that there is an $e$ such that  $M,a,c\equiv_{\rm qf}N,d,e$. Then elimination of quantifiers follows by back-and-forth. Let $p(z,x)=\eqntp(a,c)$, if $p(a,x)$ is trivial, then any $e$ in $N$ realizing $o(x/d)$ satisfies $M,a,c\equiv_{\rm qf}N,d,e$. In this case it suffices to observe that $\omega$-saturation ensures that $o(x/d)$ is realized in $N$.  Otherwise, since $p(a,x)$ is prime over $a$, by Observation~\ref{observation2}, there is an equational formula $\xi(z,x)$ such that $\vdash_{\!a}\xi(a,x)\iff p^\bullet(a,x)\wedge p(a,x)$. As $M,a\equiv_{\rm qf}N,d$, the same holds with $d$ substituted for $a$. Let $\theta(z)$ be as in \ssf{D1}. Then, for every $b$ that satisfies $\theta(z)$, the formula  $\xi(b,x)$ is consistent, so  $\xi(b,x)$ satisfied in any existentially closed model containing $b$. Then the formula $\A z[\theta(z)\imp\E x\,\xi(z,x)]$ belongs to $T_c$, so it holds in $N$. Finally, as $M,a\equiv_{\rm qf}N,d$ we obtain that $N\models\xi(d,e)$ for some $e$. As $\xi(d,x)$ is maximal over $d$, we obtain $M,a,c\equiv_{\rm qf}N,d,e$ as required.
\end{proof}

\begin{corollary}\label{stronglyminimal}
Let $T$ be a Krull-minimal theory. Then the model-companion of $T$ is a strongly minimal theory.
\end{corollary}

\begin{proof}
By Theorem~\ref{minimal} and~\ref{richsaturated} every model of the model-companion of $T$ is minimal.
\end{proof}

\section{Krull dimension}\label{noetherianity}

Let $p(z)$ be a consistent equational $A$-type. Define \emph{$\kdim(p(z)/A)$}, the \emph{Krull dimension\/} of $p(z)$ over $A$, as the maximal $n$ such that there are some prime equational $A$-types $p_0(z),\dots,p_n(z)$ such that $p_i(z)\vdash_{\!A} p_{i+1}(z)\nvdash_{\!A} p_i(z)$ and $p_n(z)\vdash_A p(z)$. Directly form the definition we obtain that, if $p(z)$ is a prime equational type, then $\kdim(p(z)/A)=0$ if and only if $p(z)$ is maximal. Then, by Fact~\ref{juisn}, for any equational type: $\kdim(p(z)/A)=0$ if and only if $p(z)$ is a disjunction of maximal types.

We want to prove that $\kdim(p(z)/A)$ is bounded by the length of $z$. First we introduce another natural dimension $\odim(p(z)/A)$. This is, roughly, the maximal degree of transcendence of a tuple satisfying $p(z)$. We will prove that $\kdim(p(z)/A)\le\odim(p(z)/A)$.

Recall that the consistency of $o(z/A)$ is equivalent to the primality over $A$ of the trivial type $z=z$, that is, to requiring the validity of the following implication for every pair of equational $A$-formulas $\phi(z)$ and $\psi(z)$:
\begin{itemize}
\item[d]\hfil$\vdash_{\!A}\phi(z)\vee\psi(z)\ \ \ \IMP\ \ \ \vdash_{\!A}\phi(z)\ \textrm{ or }\ \vdash_{\!A}\psi(z)$
\end{itemize}

If $b$ is a tuple that realizes $o(z/A)$, we say that \emph{$b$ is transcendental over $A$}. So, $b$ is transcendental over $A$ if for every equational $A$-formula $\vdash_{\!A,b}\phi(b)\iff\phi(z)$.

\begin{fact}\label{dp}
Let $T$ be a locally Krull-minimal theory and let $A$ be an arbitrary set of parameters. The following facts hold.
\begin{itemize}
\item[a] $o(z/A)$ is consistent for every tuple of variables $z$;
\item[b] $o(z/A)\subseteq o(z/B)$ whenever $A\subseteq B$;
\item[c] $o(z/A)$ is non-trivial whenever $A$ is non-empty or $z$ is a tuple of length $>1$.
\end{itemize}
\end{fact}
\begin{proof}
To prove \ssf{a} we proceed by induction on the length of $z$. Assume that \ssf{d} above holds for tuple $z$ and prove it holds for the tuple $x,z$. Suppose $\vdash_{\!A}\phi(x,z)\vee\psi(x,z)$. Then, for $a$ arbitrary, $\vdash_{\!A,a}\phi(a,z)\vee\psi(a,z)$ so, from the induction hypothesis, either $\vdash_{\!A,a}\phi(a,z)$ or $\vdash_{\!A,a}\psi(a,z)$. Now, let $b$ be arbitrary and let $a$ be transcendental over $A,b$. Suppose for definiteness that $\vdash_{\!A,a}\phi(a,z)$ obtains. Then $o(x/A,b)\vdash_{\!A}\phi(x,z)$ so, by compactness, $\vdash_{\!A,b}\xi(x)\vee\phi(x,z)$ for some equational $A,b\,$-formula $\xi(x)$ such that $\nvdash_{\!A,b}\xi(x)$. Then,  $\vdash_{\!A,b}\xi(x)\vee\phi(x,b)$ so, by \ssf{D0}, either $\vdash_{\!A,b}\xi(x)$ or $\vdash_{\!A,b}\phi(x,b)$. The first is contrary to the choice of $\xi(x)$, so $\vdash_{\!A,b}\phi(x,b)$. Finally, by the arbitrarity of $b$, we conclude  $\vdash_{\!A}\phi(x,z)$.

Claim \ssf{b} is consequence of amalgamation, in fact, from Theorem~\ref{AP} it follows that any quantifier-free formula consistent over $A$ is consistent over any $B$ containing $A$. Finally, to prove \ssf{c} observe that $\nvdash z_1= z_2$ and that $\nvdash x=a$ for any $a\in A$.
\end{proof}

\textbf{\boldmath In the sequel we work over a Krull-minimal theory $T$ and by $A$ we denote an arbitrary set of parameters. Alternatively, one can assume that $T$ only locally Krull-minimal and that $A$ is finite. The letter $z$ denotes the tuple $z_0,\dots,z_{n-1}$.}

Let $I=\{i_1,\dots,i_k\}$ for some $0\le i_1<\dots<i_k<n$. We write $z_I$ for the tuple $z_{i_1},\dots,z_{i_k}$. Let $p(z)$ be a consistent $A$-type. Define \emph{$\odim(p(z)/A)$}, which we call the \emph{algebraic dimension\/} of $p(z)$ over $A$ to be the largest cardinality of some $I$ such that $o(z_I/A) \wedge p(z)$ is consistent. We agree that when $I$ is empty $o(z_I/A)$ is trivial so $\odim(p(z)/A)$ always defined and $0\le\odim(p(z)/A)\le n$.

\begin{observation}\label{maxdim}
Let $q(z)$ be an $A$-type and let $P$ be any set of $A$-types such that

\hfil$\displaystyle\vdash_{\!A}\ q(z)\ \iff \bigvee_{p(z)\in P}p(z)$.

Then $\odim(q(z)/A)=\max\{ \odim(p(z)/A)\ :\ p(z)\in P\}$. In particular, if $q(z)$ is equational, by Fact~\ref{juisn}, its dimension is the maximal dimension of a prime equational $A$-type $p(z)$ such that $p(z)\vdash_{\!A}q(z)$.\QED
\end{observation}

Now, let $p(z,w)$ be an arbitrary $A$-type. Define

\hspace*{14ex}\llap{\emph{$p(z,\mbox{-})$}}\ \ :=\ \ $\Big\{\phi(z)\ :\ \phi(z)\textrm{ is an equational }A\textrm{-formula and }  p(z,x)\vdash_{\!A}\phi(z)\Big\}$.

Note the dependency on $A$, however, as $A$ will be always clear from the context, we omit it from the notation. Note also that $p(z,\mbox{-})$ is by definition an equational type, independently of the complexity of $p(z,w)$.

\begin{lemma}\label{lksihn}
Let $p(z)$ be a non-trivial equational $A$-type. Let $I\subseteq\{0,\dots,n-1\}$ be a set of cardinality $\odim(p(z)/A)$ such that $o(z_I/A) \wedge p(z)$ is consistent. Then 

\hfil$\displaystyle o(z_I/A)\ \vdash_{\!A}\ p(z)\iff\bigvee^k_{h=1}\xi_h(z)$ 

for some equational $A$-formulas $\xi_h(z)$ such that $o(z_I/A) \wedge\xi_h(z)$ is maximal over $A$.
\end{lemma}
\begin{proof}
Observe that by Fact~\ref{nonprime}, if $q(z)$ is maximal and $q(z)\nvdash_{\!A} p(z,x)$ then 

\hfil$\displaystyle q(z)\ \vdash_{\!A}\ p(z,x)\iff\bigvee^k_{h=1}\xi_h(z,x)$

for some equational formulas $\xi_h(z,x)$ such that $q(z)\wedge\xi_h(z,x)$ is maximal. This will be used below. Now, let $I$ be as in the statement of the lemma and let $J$ be maximal such that $I\subseteq J\subseteq \{0,\dots,n-1\}$ and

\hfil$\displaystyle o(z_I/A)\ \vdash_{\!A}\ p(z_J,\mbox{-})\iff\bigvee^k_{h=1}\xi_h(z_J)$ 

for some for some equational formulas $\xi_h(z_J)$ such that $o(z_I/A) \wedge\xi_h(z_J)$ is maximal. Note that if we take $I=J$ then $p(z_J,\mbox{-})$ is a consequence of $o(z_I/A)$ and the requirement is satisfied with $\xi_h(z_J)$ trivial. So the required $J$ exists. The lemma follows if we show that $J=\{0,\dots,n-1\}$. Suppose not and let $m\in\{0,\dots,n-1\}\sm J$. Let $1\le h\le k$ be arbitrary. We claim that $o(z_I/A) \wedge\xi_h(z_J)\ \nvdash_{\!A}\ p(z_J,z_m,\mbox{-})$: if not $p(z_J,z_m,\mbox{-})$ would be consistent with $o(z_I,z_m/A)$. Then also $p(z)$ would be consistent with $o(z_I,z_m/A)$ contradicting the maximality of $I$. So we can apply the observation above

\hfil$\displaystyle o(z_I/A) \wedge\xi_h(z_J)\ \vdash_{\!A}\ p(z_J,z_m,\mbox{-})\iff\bigvee^{k_h}_{l=1}\xi_{h,l}(z_J,z_m)$

for some equational formulas $\xi_{h,l}(z)$ such that $o(z_I/A)\wedge\xi_h(z_J) \wedge\xi_{h,l}(z_J,z_m)$ is maximal. Then it follows that

\hfil$\displaystyle o(z_I/A)\ \vdash_{\!A}\ p(z_J,z_m,\mbox{-})\iff\bigvee^k_{h=1}\bigvee^{k_h}_{l=1}\xi_h(z_J) \wedge\xi_{h,l}(z_J,z_m)$

This contradicts the maximality of $J$ proving the lemma.
\end{proof}

\begin{fact}
Suppose $A$ is non-empty. Let $p(z)$ be a non-trivial equational $A$-type. Then the following are equivalent
\begin{itemize} 
\item[a] $\odim(p(z)/A)=0$;
\item[b] $p(z)$ is a disjunction of finitely many maximal equational formulas;
\item[c] $p(z)$ is a disjunction of possibly infinitely many maximal equational types.
\end{itemize}
\end{fact}

\begin{proof}
To prove \ssf{a}$\IMP$\ssf{b}, suppose $\odim(p(z)/A)=0$. As $o(z_I/A)$ is trivial when $I$ is empty, Lemma~\ref{lksihn} implies that $p(z)$ is a disjunction a maximal formulas. The implication \ssf{b}$\IMP$\ssf{c} is trivial. To prove \ssf{c}$\IMP$\ssf{a}, assume \ssf{c} and suppose for a contradiction that $o(z_i/A)\wedge p(z)$ is consistent for some $0\le i<n$. Then $o(z_i/A)$ is consistent with some maximal equational type $q(z)\vdash p(z)$. Then $q(z_i,\mbox{-})$ is maximal equational type consistent with $o(z_i/A)$. It follows that $q(z_i,\mbox{-})$ is trivial. Then also $o(z_i/A)$, which is equivalent to it, is trivial. This cannot be by \ssf{c} of Fact~\ref{dp}.
\end{proof}

\begin{theorem}\label{decrease}
Let $p(z)$ and $q(z)$ be non-trivial equational $A$-types such that $q(z)\vdash_{\!A} p(z)\nvdash_{\!A} q(z)$ and assume that $p(z)$ is prime. Then $\odim(q(z)/A)<\odim(p(z)/A)$.
\end{theorem}
\begin{proof}
Clearly $\odim(q(z)/A)\le\odim(p(z)/A)$. Suppose for a contradiction that equality holds. Fix $I$ of cardinality $\odim(q(z)/A)$ such that $o(z_I/A)\wedge q(z)$ is consistent. Then $o(z_I/A)\wedge p(z)$ is also consistent. It is easy to check that $o(z_I/A)\wedge p(z)$ is prime so, by Lemma~\ref{lksihn}, it is maximal. Let $\phi(z)$ be an equational formula in $q(z)$ such that $p(z)\nvdash_{\!A}\phi(z)$. As $o(z_I/A)\wedge p(z)\wedge\phi(z)$ is consistent, by maximality, $o(z_I/A)\wedge p(z)\vdash_{\!A}\phi(z)$, so $p(z)\vdash_{\!A}\phi(z)\vee\psi(z_I)$ where $\psi(z_I)$ is an equational formula whose negation is in $o(z_I/A)$. By primality, either $p(z)\vdash_{\!A}\phi(z)$ or $p(z)\vdash_{\!A}\psi(z_I)$. The first possibility contradicts the choice of $\phi(z)$, the second contradicts the consistency of $o(z_I/A)\wedge p(z)$.
\end{proof}

\begin{corollary}\label{finitekrull}
From the theorem it follows that $\kdim(p(z)/A)\le\odim(p(z)/A)$, so equational types have a finite Krull-dimension.\QED
\end{corollary}

\section{Final remarks and questions}

The main question is whether (locally) Krull-minimal theories are (locally) Noetherian this meaning that every equational $A$-types $p(z)$ is principal (where $A$ is finite in the local case). A second question is whether Krull and algebraic dimension agree.

\begin{theorem}\label{k=o}
Suppose that, for any set of parameters $B\supseteq A$, there is no equational type such that $q(z,x)\vdash_{\!B} o(x/B)$. Then $\kdim(p(z)/A)=\odim(p(z)/A)$ for every equational $A$-type $p(z)$.
\end{theorem}

\begin{proof}
By Corollary~\ref{finitekrull} we have $\kdim(p(z)/A)\le\odim(p(z)/A)$, so we only need to prove the converse inequality. As observed in~\ref{maxdim}, there is a prime type such that $p'(z)\vdash_{\!A}p(z)$ and $\odim(p'(z)/A)=\odim(p(z)/A)$. As clearly $\kdim(p'(z)/A)\le\kdim(p(z)/A)$, it suffices to prove the inequality $\odim(p'(z)/A)\le\kdim(p'(z)/A)$. We suppose $\odim(p'(z)/A)=m+1$ and show that there is an equational $A$-type $q(z)$ such that $q(z)\vdash_{\!A}p'(z)\nvdash_{\!A}q(z)$ and $\odim(q(z)/A)=m$. The theorem follows by induction. Let $I\subseteq\{0,\dots,n-1\}$ and $i\in\{0,\dots,n-1\}\sm I$ be such that $o(z_I,z_i/A)$ is consistent with $p'(z)$. Observe that $o(z_I/A)\wedge p'(z)\nvdash_{\!A} o(z_I,z_i/A)$. This follows from the hypothesis above after replacing $z_I$ with some parameters $b_I$. Then there is an equational $A$-formula $\psi(z_I,z_i)$ such that $\nvdash_{\!A}\psi(z_I,z_i)$ and $o(z_I/A)\wedge p'(z)\wedge \psi(z_I,z_i)$ is consistent. Then $m\le \odim(p'(z)\wedge \psi(z_I,z_i)/A)$. As $ p'(z)$ is consistent with $o(z_I,z_i/A)$ while $p'(z)\wedge \psi(z_I,z_i)$ is not, then $p'(z)\nvdash\psi(z_I,z_i)$. So $p'(z)\wedge \psi(z_I,z_i)\vdash_{\!A}p'(z)\nvdash_{\!A}p(z)\wedge \psi(z_I,z_i)$ and from Lemma~\ref{decrease} we obtain $\odim(p'(z)\wedge \psi(z_I,z_i)/A)< \odim(p'(z)/A)$. So $\odim(p'(z)\wedge \psi(z_I,z_i)/A)=m$ as required.
\end{proof}

When $T$ is Noetherian and the model-companion of $T$ is not $\omega$-categorical the hypothesis of theorem above is satisfied. In fact, suppose $\zeta(z,x)\vdash_{\!B} o(x/B)$ for some equational $B$-formula and let $U$ be an existentially closed saturated model containing $B$. By elimination of quantifier there is a pair of equational $B$-formulas $\phi(x)$ and $\psi(x)$ such that $\phi(x)\wedge\neg\psi(x)\imp o(x/B)$ holds in $U$. Observe that $\phi(x)$ has to be trivial, otherwise $o(x/B)$ would contain $\neg\phi(x)$. Then $o(x/B)$ is principal and, together with \ssf{D3}, this implies that there are only finitely many quantifier-free $B$-types in $x$.


\vfil
Domenico Zambella\\
Dipartimento di Matematica\\
Universit\`a di Torino\\
via Carlo Alberto 10\\
10123 Torino\\
Italy

\begin{thebibliography}{XX}%

\bibitem[1]{m1} Marker, D.: {\it Model theory: an introduction}. New York. Springer-Verlag 2002

\bibitem[2]{m2} Marker, D.: {\it Introduction to the model theory of fields}, in {\it Model Theory of Fields}, D. Marker, M. Messmer and A. Pillay ed., Lecture Notes in Logic vol.5, Association for Symbolic Logic (2006)

\bibitem[3]{j} Junker, M.: {\it  A note on equational theories}, J. Symbolic Logic vol.65 no.4 1705--1712 (2000)

\bibitem[4]{ps} Pillay, A.,  Srour, G.: {\it   Closed sets and chain conditions in stable theories.}  J. Symbolic Logic vol.49   no.4 1350--1362 (1984)
\end{thebibliography}
\end{document}